\documentclass[12pt,a4paper,twoside]{article}
\usepackage{hyperref}
\usepackage{amsmath,amsfonts,amsthm,amsopn,color,amssymb}
\usepackage{palatino}
\usepackage{graphicx}

\newcommand{\eps}{\varepsilon}

\newcommand{\R}{\mathbb{R}}

\newcommand{\RN}{{\mathbb{R}^N}}
\newcommand{\RT}{{\mathbb{R}^3}}

\renewcommand{\le}{\leqslant}
\renewcommand{\ge}{\geqslant}
\renewcommand{\a }{\alpha }

\renewcommand{\d }{\delta }

\newcommand{\g }{\gamma }

\renewcommand{\l }{\lambda}
\newcommand{\n }{\nabla }

\renewcommand{\t}{\theta}

\newcommand{\G}{\Gamma}

\newcommand{\Itql}{I^T_{q,{\l}}}
\newcommand{\Itqln}{I^T_{q,{\l_n}}}
\newcommand{\Itq}{I^T_{q}}

\newcommand{\Iq}{I_{q}}

\renewcommand{\H}{H^1(\RT)}
\newcommand{\Hr}{H^1_r(\RT)}

\newcommand{\E}{\mathcal{E}}

\newcommand{\D }{{\mathcal D}^{1,2}(\RT)}
\newcommand{\irn }{\int_{\RN}}
\newcommand{\irt }{\int_{\RT}}

\def\bbm[#1]{\mbox{\boldmath $#1$}}

\newtheorem{theorem}{Theorem}[section]
\newtheorem{lemma}[theorem]{Lemma}

\renewenvironment{proof}{\noindent{\textbf{Proof\quad}}}{$\hfill\square$\vspace{0.2 cm}\\}
\newenvironment{proofmain}{\noindent{\textbf{Proof of Theorem  \ref{main}\quad}}}{$\hfill\square$\vspace{0.2 cm}\\}

\title{{\bf On the Schr\"odinger-Maxwell equations\\
under the effect of a general\\
nonlinear term\footnote{The authors are supported by M.I.U.R. -
P.R.I.N. ``Metodi variazionali e topologici nello studio di
fenomeni non lineari''}}}

\author{A. Azzollini \thanks{Dipartimento di Matematica ed Informatica, Universit\`a degli
Studi della Basilicata,  Via dell'Ateneo Lucano 10, I-85100
Potenza, Italy, e-mail: {\tt antonio.azzollini@unibas.it}}
 \; \& \;
P. d'Avenia\thanks{Dipartimento di Matematica, Politecnico di
Bari, Via E. Orabona 4, I-70125 Bari, Italy, e-mail: {\tt
p.davenia@poliba.it}}
 \; \& \;
A. Pomponio\thanks{Dipartimento di Matematica, Politecnico di
Bari, Via E. Orabona 4, I-70125 Bari, Italy, e-mail: {\tt
a.pomponio@poliba.it}}}

\date{}

\begin{document}

\maketitle

\begin{abstract}
In this paper we prove the existence of a nontrivial solution to
the nonlinear Schr\"odinger-Maxwell equations in $\RT,$ assuming
on the nonlinearity the general hypotheses introduced by
Berestycki \& Lions.
\end{abstract}

\section{Introduction}

In the recent years, the following electrostatic nonlinear Schr\"odinger-Max\-well  equations, also known as nonlinear Schr\"odinger-Poisson system,
\begin{equation}    \label{SM}\tag{${\cal SM}$}
\left\{
\begin{array}{ll}
-\Delta u+q\phi u=g(u)&\hbox{in }\RT,
\\
-\Delta \phi=q u^2&\hbox{in }\RT,
\end{array}
\right.
\end{equation}
have been object of interest for many authors. Indeed a similar
system arises in many mathematical physics contexts, such as in quantum electrodynamics, to describe the
interaction between a charge particle interacting with the
electromagnetic field, and also in semiconductor theory,
in nonlinear optics and in plasma physics. We refer to \cite{BF} for more details in the physics aspects.
\\
The greatest part of the literature focuses on the study of the
previous system for the very special nonlinearity
$g(u)=-u+|u|^{p-1}u$, and existence, nonexistence and multiplicity
results are provided in many papers for this particular problem
(see \cite{AR,AP08,C2,DM1,DM2,DA,JZ,K1,K2,Ru,ZZ}). In
\cite{C1,CG,WZ}, also the linear and the asymptotic linear case
have been studied, whereas in \cite{PiS} the problem has been
dealt with in a bounded domain, with Neumann conditions on the
boundary.\\
The aim of this paper is to study the Schr\"odinger-Maxwell system
assuming the same very general hypotheses introduced by Berestycki
\& Lions, in their celebrated paper \cite{BL1}. Actually, we
assume that the following hold for $g$:
\begin{itemize}
\item[({\bf g1})] $g\in C(\R,\R)$;
\item[({\bf g2})]
$-\infty <\liminf_{s\to 0^+} g(s)/s\le \limsup_{s\to 0^+}
g(s)/s=-m<0$;
\item[({\bf g3})] $-\infty \le\limsup_{s\to +\infty}
g(s)/s^5\le 0$;
\item[({\bf g4})] there exists $\zeta>0$
such that $G(\zeta):=\int_0^\zeta g(s)\,d s>0$.
\end{itemize}

Using similar assumptions on the nonlinearity $g$, \cite{AP,JT} and \cite{PS} studied, respectively, a nonlinear Schr\"odinger equation in presence of an external potential and a system of weakly coupled nonlinear Schr\"odinger equations. We mention also \cite{BF2} where the Klein-Gordon and in Klein-Gordon-Maxwell equations are considered.

The main result of the paper is
\begin{theorem}\label{main}
If $g$ satisfies ({\bf g1-4}), then there exists $q_0>0$ such
that, for any $0<q<q_0$, problem \eqref{SM} admits a nontrivial positive
radial solution $(u,\phi)\in \H\times \D$.
\end{theorem}
Some remarks on this result are in order:

\begin{itemize}
\item the assumptions are trivially satisfied by nonlinearities like $g(u)= - u +
|u|^{p-1}u$, for any $p\in ]1,5[$;
\item hypotheses on $g$ are {\it almost necessary} in
the sense specified in \cite{BL1};
\item the fact that the result is
obtained for small $q$ is not surprising for at least two
reasons: first, because small $q$ makes, in some sense,
less strong the influence of the term $\phi u$,
which constitutes, in the first equation, a perturbation
with respect to the classical nonlinear Schr\"odinger
equation treated in \cite{BL1}; second, there is a nonexistence result for large $q$ and $g(u)= - u +
|u|^{p-1}u$ with $p\in ]1,2]$ (see \cite{Ru}).
\end{itemize}

From the technical point of view, dealing with \eqref{SM} under
the effect of a general nonlinear term presents several
difficulties.
Indeed the lack of the following global Ambrosetti-Rabinowitz growth hypothesis on $g$:
\[
\hbox{there exists } \mu>2 \hbox{ such that } 0<\mu G(s)\le g(s)s, \hbox{ for all $s\in \R$},
\]
brings on two obstacles to the standard Mountain Pass arguments both in checking the geometrical assumptions in the
functional and in proving the boundedness of its Palais-Smale
sequences. To overcome these difficulties, we will use a combined
technique consisting in a truncation argument (see \cite{JC,K2}) and a monotonicity
trick {\it \`a la} Jeanjean~\cite{J} (see also Struwe \cite{Sw}).

The paper is organized as follows. In Section \ref{fun} we
introduce the functional framework for solving our problem by a
variational approach. In Section \ref{per} we define a sequence of
modified functionals on which we can easily apply the Mountain
Pass Theorem. Then we study the convergence of the sequence of
critical points obtained. Finally the Appendix is devoted to prove
a Pohozaev type identity which we use, in Section \ref{per}, as a
fundamental tool in our arguments.

\vspace{0.5cm}
\begin{center}
{\bf NOTATION}
\end{center}

\begin{itemize}
\item For any $1\le s\le +\infty$, we denote by $\|\cdot\|_s$ the usual norm of the Lebesgue space $L^s(\RT)$;
\item $\H$ is the usual Sobolev space endowed with the norm
\[
\|u\|^2:=\irt |\n u|^2+ u^2;
\]
\item $\D$ is completion of $C_0^\infty(\RT)$ (the compactly
supported functions in $C^\infty(\RT)$) with respect to the norm
\[
\|u\|_{\D}^2:=\irt |\n u|^2;
\]
\item for brevity, we denote $\a=12/5$.
\end{itemize}


\section{Functional setting}\label{fun}

We first recall the following well-known facts (see, for instance \cite{BF,BFMP,DM1,Ru}).
\begin{lemma}\label{le:prop}
For every $u\in \H$, there
exists a unique $\phi_u\in \D$ solution of
\[
-\Delta \phi=q u^2,\qquad \hbox{in }\RT.
\]
Moreover
\begin{itemize}
\item[i)] $\|\phi_u\|^2_{\D}=q\irt\phi_u u^2$;
\item[ii)] $\phi_u\ge 0$;
\item[iii)] for any $\t>0$: $\phi_{u_\t}(x)=\t^2\phi_u(x/\t)$,
where $u_\t(x)=u(x/\t )$;
\item[iv)] there exist $C,C'>0$ independent of $u\in\H$ such that
$$\|\phi_u\|_{\D}\le C q \|u\|^2_\a,$$
and
\begin{equation}\label{eq:phiq}
\irt\phi_u u^2\le C'q \|u\|^4_\a;
\end{equation}
\item[v)] if $u$ is a radial function then $\phi_u$ is radial, too.
\end{itemize}
\end{lemma}
    
Following \cite{BL1}, define $s_0:=\min\{s\in
[\zeta,+\infty[\;\mid g(s)=0\}$ ($s_0=+\infty$ if $g(s)\neq 0$ for
any $s\ge\zeta$). We set $\tilde g:\R\to\R$ the function such that
    \begin{equation}\label{eq:tilde}
      \tilde g(s)=\left\{
        \begin{array}{ll}
                g(s) &\hbox{ on } [0,s_0];
                \\
                0 &\hbox{ on } \R_+\setminus [0,s_0];
                \\
                (g(-s)-ms)^+ - g(-s) &\hbox{ on } \R_-.
      \end{array}
      \right.
    \end{equation}
By the strong maximum principle and by ii) of Lemma \ref{le:prop}, if $u$ is a nontrivial solution of \eqref{SM} with
$\tilde g$ in the place of $g$, then $0< u< s_0$ and so it is a positive solution of \eqref{SM}. Therefore
we can suppose that $g$ is defined as in \eqref{eq:tilde}, so that
({\bf g1}), ({\bf g2}), ({\bf g4}) and then the following limit
    \begin{equation}\label{eq:limg}
        \lim_{s\to\pm\infty} \frac{g(s)}{s^{5}}=0
    \end{equation}
hold. 
\\
We set
\begin{align*}
g_1(s) &:=
\left\{
\begin{array}{ll}
(g(s)+ms)^+, & \hbox{if }s\ge0,	
\\
0, & \hbox{if }s<0,	
\end{array}
\right.
\\
g_2(s) &:=g_1(s)-g(s), \quad \hbox{for }s\in \R.
\end{align*} 
Since
    \begin{align}
        \lim_{s\to 0}\frac{g_1(s)}{s} &= 0,\nonumber
        \\
        \lim_{s\to\pm\infty}\frac{g_1(s)}{s^5}&=0,\label{eq:lim2}
    \end{align}
and
    \begin{equation}
        g_2(s) \ge ms,\quad  \forall s\ge 0,\label{eq:g2}
    \end{equation}
by some computations, we have that for any $\eps>0$ there exists
$C_\eps>0$ such that
    \begin{equation}
        g_1(s) \le C_\eps s^5+\eps g_2(s),\quad  \forall
        s\ge0\label{eq:g1g2}.
    \end{equation}
If we set
    \begin{equation*}
        G_i(t):=\int^t_0g_i(s)\,ds,\quad i=1,2,
    \end{equation*}
then, by \eqref{eq:g2} and \eqref{eq:g1g2}, we have
    \begin{equation}
         G_2(s) \ge \frac m 2 s^2,\quad  \forall s\in\R\label{eq:G2}
    \end{equation}
and for any $\eps>0$ there exists $C_\eps>0$ such that
    \begin{equation}
        G_1(s) \le \frac {C_\eps} {6} s^6+\eps G_2(s),\quad  \forall
        s\in\R\label{eq:G1G2}.
    \end{equation}

The solutions $(u,\phi)\in \H \times \D$ of
\eqref{SM} are the critical points of the action functional $\mathcal{E}
\colon \H \times \D \to \R$, defined as
\[
\mathcal{E}_q(u,\phi):=\frac 12 \irt |\n u|^2
-\frac 14 \irt |\n \phi|^2
+\frac q2 \irt \phi u^2
-\irt G(u).
\]

The action functional $\E_q$ exhibits a strong indefiniteness, namely it is
unbounded both from below and from above on infinite dimensional
subspaces. This indefiniteness can be removed using the reduction
method described in \cite{BF,BFMP}, by which we are led to study a
one variable functional that does not present such a strongly
indefinite nature.
Hence, it can be proved that $(u,\phi)\in H^1(\RT)\times \D$ is a
solution of \eqref{SM} (critical point of functional $\mathcal{E}_q$) if and only if $u\in\H$ is a critical point
of the functional $I_q\colon \H\to \R$ defined as
\begin{equation*}
I_q(u)= \frac 12 \irt |\n u|^2 + \frac q4 \irt \phi_u u^2
-\irt G(u),
\end{equation*}
and $\phi=\phi_u$.

We will look for critical points of $I_q$ on $\Hr:=\{u\in\H\mid u \hbox{ is radial}\}$, which is a natural constraint.

\section{The perturbed functional}\label{per}

Kikuchi, in \cite{K2}, considered \eqref{SM}, where $g(u)=-u+|u|^{p-1}u $, with $1<p<5$. To overcome the difficulty in finding bounded Palais-Smale sequences for the associated functional $I_q$,
following \cite{JC}, he introduced the cut-off function $\chi\in C^\infty(\R_+,\R)$ satisfying
\begin{equation}\label{eq:defchi}
    \left\{
\begin{array}{ll}
    \chi(s)=1,&\hbox{for }s\in[0,1],\\
    0\le \chi(s)\le 1,&\hbox{for }s\in]1,2[,\\
    \chi(s)=0,&\hbox{for }s\in[2,+\infty[,\\
    \|\chi '\|_\infty \le 2,&
\end{array}
    \right.
\end{equation}
and studied the following modified functional $\widetilde{\Itq}:\H\to \R$
\begin{equation*}
\widetilde{\Itq}(u)=\frac 12 \irt |\n u|^2 + \frac{q}{4} \widetilde{k}_T (u)\irt \phi_u u^2
 - \irt G(u),
\end{equation*}
where, for every $T>0$,
\[
\widetilde{k}_T(u)=\chi\left(\frac{\|u\|^2}{T^2} \right).
\]
With this penalization, for $T$ sufficiently large and for $q$ sufficiently small, he is able to find a critical point $\bar u$ such that $\|\bar u\|\le T$ and so he concludes that $\bar u$ is a critical point of $I_q$.

Let us observe that if $g(u)=f(u)-u$ with $f$ satisfying the Ambrosetti-Rabinowitz growth condition, the arguments of Kikuchi still hold with slide modifications.

On the other hand, in presence of nonlinearities satisfying Berestycki-Lions assumptions, further difficulties arise about the geometry of our functional and compactness. First of all, as in \cite{K2}, we introduce a similar truncated functional $\Itq :\Hr\to \R$
\begin{equation*}
\Itq(u)=\frac 12 \irt |\n u|^2 + \frac{q}{4} k_T (u)\irt \phi_u u^2
 - \irt G(u),
\end{equation*}
where, now,
\[
k_T(u)=\chi\left(\frac{\|u\|_\a^\a}{T^\a} \right).
\]
The $C^1-$functional $\Itq$ satisfies the geometrical assumptions of the
Moun\-tain-Pass Theorem but, under our general assumptions on the
nonlinearity, we are not able to obtain the boundedness of the
Palais-Smale sequences. Therefore we use an indirect approach
developed by Jeanjean. We apply the following slight modified
version of \cite[Theorem~1.1]{J} (see \cite{J2}).

    \begin{theorem}\label{th:J}
        Let $\big(X,\|\cdot\|\big)$ be a Banach space and
        $J\subset\R_+$ an interval.
Consider the family of $C^1$ functionals on $X$
    \begin{equation*}
        I_\l(u)=A(u)- \l B(u),\quad\forall\l\in J,
    \end{equation*}
with $B$ nonnegative and either $A(u)\to + \infty$ or
$B(u)\to+\infty$ as $\|u\|\to\infty$ and such that $I_\l(0)=0$.\\
For any $\l\in J$ we set
    \begin{equation*}
        \Gamma_\l:=\{\gamma\in C([0,1],X)\mid \gamma(0)=0,
        I_\l(\gamma(1))< 0\}.
    \end{equation*}
If for every $\l\in J$ the set $\G_\l$ is nonempty and
    \begin{equation}\label{eq:cl}
        c_\l:=\inf_{\gamma\in\Gamma_\l}\max_{t\in[0,1]}
        I_\l(\gamma(t)) >0,
    \end{equation}
then for almost every $\l\in J$ there is a sequence $(v_n)_n
\subset X$ such that
    \begin{itemize}
        \item[(i)] $(v_n)_n$ is bounded;
        \item[(ii)] $I_\l(v_n)\to c_\l$;
        \item[(iii)] $(I_\l)'(v_n)\to 0$ in the dual $X^{-1}$ of $X$.
    \end{itemize}
    \end{theorem}
In our case, $X=\Hr$,
    \begin{align*}
        A(u) &:=\frac 12 \irt |\n u|^2 + \frac{q}{4} k_T (u)\irt \phi_u u^2
        +   \irt G_2(u),\\
        B(u) &:=\irt G_1(u),
    \end{align*}
so that the perturbed functional we study is
    \begin{equation*}
        \Itql(u)=\frac 12 \irt |\n u|^2 + \frac{q}{4} k_T (u)\irt \phi_u
        u^2 +\irt G_2(u)
        -\l \irt G_1(u).
    \end{equation*}

Actually, this functional is the restriction to the radial functions of a $C^1$-functional defined on the whole space $\H$ and  for every $u,v\in\H$
    \begin{multline*}
        \langle(\Itq)'(u),v\rangle=\irt (\n u|\n v) +q k_T(u)\irt\phi_u uv \\
        +\frac{q\a}{4 T^\a}\chi'\left(\frac{\|u\|_\a^\a}{T^\a}\right)\irt\phi_{u} u^2 \irt|u|^{\a-2}uv
        +\irt  g_2(u)v - \l\irt  g_1(u)v.
    \end{multline*}

In order to apply Theorem \ref{th:J}, we have just to define a
suitable interval $J$ such that $\Gamma_\l\neq\emptyset$, for any $\l\in J$,
and \eqref{eq:cl} holds.\\
Observe that, according to \cite{BL1}, as a consequence of ({\bf g4}), there exists a function
$z\in\Hr$ such that
\begin{equation}\label{eq:zeta}
\irt G_1(z)-\irt G_2(z)=\irt G(z)>0.
\end{equation}
Then there exists $0<\bar\d<1$ such that
    \begin{equation}\label{eq:d}
        \bar\d\irt G_1(z)-\irt G_2(z)>0.
    \end{equation}
We define $J$ as the interval $[\bar \delta , 1].$

\begin{lemma}\label{le:Gamma}
$\Gamma_\l\neq\emptyset$, for any $\l\in J$.
\end{lemma}

\begin{proof}
Let $\l\in J$. Set $\bar\t>0$ and $\bar z=z(\cdot /\bar\t)$.\\
Define $\g:[0,1]\to\Hr$ in the following way
    \begin{equation*}
        \g(t)=\left\{\begin{array}{ll}
            0,& \hbox{if }t=0,
            \\
            \bar z(\cdot/t), & \hbox{if }0<t\le1.
        \end{array}
        \right.
    \end{equation*}
It is easy to see that $\g$ is a continuous path from $0$ to $\bar
z.$ Moreover, we have that
    \begin{multline*}
        \Itql(\g(1)) \le \frac {\bar\t} 2 \irt |\n z|^2
        +\frac q 4 \bar\t^5 \chi\left(\frac{\bar\t^3\|z\|_\a^\a}{T^\a}\right)
        \irt\phi_z z^2\\
        +\bar\t^3
        \left(\irt G_2(z) - \bar\d\irt G_1(z)\right)
    \end{multline*}
and then, if $\bar\t$ is sufficiently large, by \eqref{eq:d} and \eqref{eq:defchi} we get $\Itql (\g(1))<0$.
\end{proof}

\begin{lemma}\label{le:cl}
There exists a constant $\tilde c>0$ such that $c_\l\ge\tilde c>0$
for all $\l\in J.$
\end{lemma}

\begin{proof}
Observe that for any $u\in\Hr$ and $\l\in J$, using \eqref{eq:G2}
and \eqref{eq:G1G2} for $\eps < 1$, we have
    \begin{align*}
        \Itql(u) & \ge \frac 1 2 \irt |\n u|^2 + \frac{q}{4} k_T (u)\irt \phi_u u^2
        +   \irt G_2(u)-\irt G_1(u)\\
                  & \ge \frac 1 2 \irt |\n u|^2  + (1-\eps) \frac
                  m 2 \irt u^2 - \frac{C_\eps}{6} \irt
                  |u|^{6}
    \end{align*}
and then, by Sobolev embeddings, we conclude that there exists
$\rho >0$ such that, for any $\l\in J$ and $u\in\Hr$ with $u\neq 0$ and
$\|u\|\le\rho,$ it results $\Itql(u)>0.$ In particular, for any
$\|u\|=\rho,$ we have $\Itql(u) \ge \tilde c >0.$ Now fix $\l\in J$
and $\g\in\G_\l.$ Since $\g(0)=0$ and $\Itql(\g(1))< 0$,
certainly $\|\g(1)\|
> \rho.$ By continuity, we deduce that there exists $t_\g\in
]0,1[$ such that $\|\g(t_\g)\|=\rho.$ Therefore, for any $\l\in
J,$
\begin{equation*}
c_\l\ge \inf_{\g\in\G_\l} \Itql(\g(t_\g)) \ge \tilde c >0.
\end{equation*}
\end{proof}
We present a variant of the Strauss' compactness result \cite{Str}
(see also \cite[Theorem A.1]{BL1}). It will be a fundamental tool
in our arguments:
\begin{theorem}\label{le:str}
Let $P$ and $Q:\R\to\R$ be two continuous functions satisfying
\begin{equation*}
\lim_{s\to\infty}\frac{P(s)}{Q(s)}=0,
\end{equation*}
$(v_n)_n,$ $v$ and $w$ be measurable functions from $\RN$ to $\R$,
with $z$ bounded, such that
\begin{align*}
&\sup_n\irn | Q(v_n(x))w|\,dx <+\infty,
\\ 
&P(v_n(x))\to v(x) \:\hbox{a.e. in }\RN. 
\end{align*}
Then $\|(P(v_n)-v)w\|_{L^1(B)}\to 0$, for any bounded Borel set
$B$.

Moreover, if we have also
\begin{align*}
\lim_{s\to 0}\frac{P(s)}{Q(s)} &=0,\\ 
\lim_{x\to\infty}\sup_n |v_n(x)| &= 0, 
\end{align*}
then $\|(P(v_n)-v)w\|_{L^1(\RN)}\to 0.$
\end{theorem}

In analogy with the well-known compactness result in \cite{BL2},
we state the following result

\begin{lemma}\label{le:PS}
For any $\l \in J$, each bounded Palais-Smale sequence for the functional $\Itql$ admits a convergent subsequence.
\end{lemma}

\begin{proof}
Let $\l\in J$ and $(u_n)_n$ be a bounded (PS) sequence for $\Itql$,
namely
    \begin{align}\label{eq:PS}
        &(\Itql(u_n))_n \hbox{ is bounded },\nonumber\\
        &(\Itql)'(u_n)\to 0 \hbox{ in } (\Hr)'.
    \end{align}
Up to a subsequence, we can suppose that there exists $u\in\Hr$
such that
\begin{align}
u_n\rightharpoonup u\;&\hbox{ weakly in }\Hr,\label{eq:weak}
\\
u_n\to u\;&\hbox{ in }L^p(\RT),\; 2<p<6,\label{eq:lp}
\\
u_n\to u\;&\hbox{ a.e. in }\RN.\label{eq:aeconv}
\end{align}
If we apply Theorem \ref{le:str} for $P(s)=g_i(s)$, $i=1,2,$
$Q(s)= |s|^5,$ $(v_n)_n=(u_n)_n,$ $v=g_i(u),$ $i=1,2$ and
$w\in C^\infty_0(\RN),$ by \eqref{eq:limg}, \eqref{eq:lim2} and
\eqref{eq:aeconv} we deduce that
\begin{equation*}
\irt g_i(u_n)w\to\irt g_i(u)w\quad i=1,2.
\end{equation*}
Moreover, by \eqref{eq:lp} and \cite[Lemma 2.1]{Ru}, we have
\begin{align*}
k_T(u_n)\irt\phi_{u_n} u_n w &\to  k_T(u)\irt\phi_u u w,
\\
\chi'\left(\frac{\|u_n\|_\a^\a}{T^\a}\right)\irt\phi_{u_n} u^2_n
\irt |u_n|^\frac{2}{5} u_n w&\to
\chi'\left(\frac{\|u\|_\a^\a}{T^\a}\right)\irt\phi_{u} u^2 \irt
|u|^\frac{2}{5} u w.
\end{align*}
As a consequence, by \eqref{eq:PS} and \eqref{eq:weak} we deduce
$(\Itql)'(u)=0$ and hence
\begin{multline}    \label{eq:solo}
\irt |\n u|^2
+q k_T(u)\irt\phi_u u^2
+\frac{q\a}{4 T^\a} \chi'\left(\frac{\|u\|_\a^\a}{T^\a}\right) \|u\|^\a_\a \irt\phi_{u} u^2
+\irt g_2(u)u
\\
=\l \irt  g_1(u)u.
\end{multline}
\\
By weak lower semicontinuity we have:
\begin{align}
\irt |\n u|^2 \le &\liminf_n \irt |\n u_n|^2. \label{eq:semi1}
\end{align}
Again, by \eqref{eq:lp} we have
\begin{align}
k_T(u_n)\irt\phi_{u_n} u_n^2 &\to  k_T(u)\irt\phi_u u^2, \label{eq:convk}
\\
\chi'\left(\frac{\|u_n\|_\a^\a}{T^\a}\right)\|u_n\|^\a_\a \irt\phi_{u_n} u^2_n
&\to
\chi'\left(\frac{\|u\|_\a^\a}{T^\a}\right)\|u\|^\a_\a\irt\phi_{u} u^2. \label{eq:convk'}
\end{align}
If we apply Theorem \ref{le:str} for $P(s)=g_1(s)s,$ $Q(s)= s^2+
s^6,$ $(v_n)_n=(u_n)_n,$ $v=g_1(u)u,$ and $w=1,$ by
\eqref{eq:limg}, \eqref{eq:lim2} and \eqref{eq:aeconv},
we deduce that
    \begin{align}
        \irt g_1(u_n)u_n \to \irt g_1(u)u.\label{eq:convg}
    \end{align}
Moreover, by \eqref{eq:aeconv} and Fatou's lemma
    \begin{align}
        \irt g_2(u)u\le &\liminf_n \irt g_2(u_n)u_n.\label{eq:convg2}
    \end{align}
By \eqref{eq:solo}, \eqref{eq:convk}, \eqref{eq:convk'}, \eqref{eq:convg} and \eqref{eq:convg2}, and
since $\langle (I_\l)'(u_n),u_n\rangle\to0$, we have
\begin{align}
\limsup_n \!\irt \! |\n u_n|^2
&=\limsup_n \left[\l \irt g_1(u_n)u_n
-\irt g_2(u_n)u_n \right. \nonumber
\\
&\quad \left. -q k_T(u_n)\irt\!\!\phi_{u_n} u_n^2
-\frac{q\a}{4 T^\a} \chi'\left(\frac{\|u_n\|_\a^\a}{T^\a}\right) \|u_n\|^\a_\a \irt\!\!\phi_{u_n} u_n^2
\right] \nonumber
\\
&\le\l \irt g_1(u)u - \irt g_2(u)u \nonumber
\\
&\quad -q k_T(u)\irt\phi_u u^2
-\frac{q\a}{4 T^\a} \chi'\left(\frac{\|u\|_\a^\a}{T^\a}\right) \|u\|^\a_\a \irt\phi_{u} u^2 \nonumber
\\
&=\irt |\n u|^2. \label{eq:limsup}
\end{align}
By \eqref{eq:semi1} and \eqref{eq:limsup}, we get
\begin{align}
\lim_n \irt |\n u_n|^2= &\irt |\n u|^2 \label{eq:sf1},
\end{align}
hence
\begin{equation}    \label{eq:sfg2}
\lim_n \irt g_2(u_n)u_n= \irt g_2(u)u.
\end{equation}
Since $g_2(s)s=ms^2 +h(s)$, with $h$ a positive and continuous function, by Fatou's Lemma we have
\begin{align*}
\irt h(u) \le &\liminf_n \irt h(u_n),
\\
\irt u^2 \le &\liminf_n \irt u_n^2.
\end{align*}
These last two inequalities and \eqref{eq:sfg2} imply that, up to
a subsequence,
\[
\irt  u^2 =\lim_n \irt u_n^2,
\]
which, together with \eqref{eq:sf1}, shows that $u_n \to u$ strongly in $\Hr$.
\end{proof}

\begin{lemma}\label{le:ul}
For almost every $\l\in J$, there exists $u^\l\in \Hr$, $u^\l\neq
0$, such that $(\Itql)'(u^\l)=0$ and $\Itql(u^\l)= c_\l$.
\end{lemma}

\begin{proof}
By Theorem \ref{th:J}, for almost every $\l\in J$, there exists a
bounded sequence $(u^\l_n)_n\subset\Hr$ such that
    \begin{align}
        \Itql(u^\l_n)&\to c_\l;\label{eq:conv}\\
        (\Itql)'(u^\l_n)&\to 0\;\hbox{in } (\Hr)'.\label{eq:p-s}
    \end{align}
Up to a subsequence, by Lemma \ref{le:PS}, we can suppose that
there exists $u^\l\in\Hr$ such that $u^\l_n \to u^\l$ in $\Hr$. By
Lemma \ref{le:cl}, \eqref{eq:conv} and \eqref{eq:p-s} we conclude.
\end{proof}

Therefore there exist $(\l_n)_n\subset J$ and $(u_n)_n\subset \Hr$ such that
\begin{equation}\label{eq:PM}
\Itqln(u_n)=c_{\l_n}, \qquad (\Itqln)'(u_n)=0.
\end{equation}

\begin{lemma}\label{le:T}
Let $u_n$ be a critical point for $\Itqln$ at level $c_{\l_n}$.
Then, for $T>0$ sufficiently large, there exists $q_0=q_0(T)$ such that for any $0<q<q_0$, up to a subsequence,
$\|u_n\|_{\a}\le T$, for any $n\ge 1.$
\end{lemma}

\begin{proof}
We will argue by contradiction.
\\
First of all, since $(\Itqln)'(u_n)=0$, $u_n$ satisfies the following Pohozaev
type identity
\begin{multline}\label{eq:pohofin}
\frac{1}{2}\irt|\nabla u_n|^2 + \frac{5q}{4}k_T(u_n)\irt\phi_{u_n} u_n^2
+\frac{3q}{T^\a} \chi'\left(\frac{\|u_n\|_\a^\a}{T^\a}\right)\|u_n\|_\a^\a\irt\phi_{u_n} u_n^2
\\
=3\l_n \irt G_1(u_n) -3 \irt G_2(u_n)
\end{multline}
(see Appendix for the proof).
\\
Moreover, combining
\eqref{eq:pohofin} with the first of \eqref{eq:PM} and by \eqref{eq:phiq}, we get
\begin{align}\label{eq:ineq}
\irt |\n u_n|^2 & =3 c_{\l_n} + \frac q 2
k_T(u_n)\irt\phi_n u_n^2+\frac{3q}
{T^\a}\chi'\left(\frac{\|u_n\|_\a^\a}{T^\a}\right)\|u_n\|_\a^\a\irt\phi_nu^2_n\nonumber\\
& \le 3 c_{\l_n}+C_1q^2k_T(u_n)\|u_n\|_\a^4 +
C_2 \chi'\left(\frac{\|u_n\|_\a^\a}{T^\a}\right)\frac {q^2}{T^\a}\|u_n\|_\a^{4+\a}.
\end{align}
We are going to estimate the right part of the previous
inequality.
By the min-max definition of the Mountain Pass level, we
have
\begin{align*}
c_{\l_n} & \le
\max_\t\Itqln\left(z\left(\cdot/\t\right)\right)
\\
& \le \max_\t\left\{\frac \t 2 \irt |\n z|^2+\t^3\left(\irt G_2(z)-\bar \d\irt
G_1(z)\right)\right\}
\\
&\qquad +\max_\t\left\{\frac q 4 \t^5 \chi \left(\frac{\t^3\|z\|_\a^\a}{T^\a}\right)\irt \phi_z z^2\right\}\\
& = A_1 + A_2(T)
\end{align*}
where $z$ is the function such that \eqref{eq:zeta} holds.
\\
Now, if $\t^3\ge 2 T^\a/\|z\|^\a_\a$ then $A_2(T)=0,$ otherwise we compute
\begin{align*}
A_2(T)\le \frac q 4
\left({\frac{2T^\a}{\|z\|_\a^\a}}\right)^{\frac 5
3}\irt \phi_z z^2= C_3q T^4.
\end{align*}
We also have
\begin{align*}
C_1q^2k_T(u_n)\|u_n\|_\a^4  \le C_4 q^2 T^4\\
C_2\chi'\left(\frac{\|u_n\|_\a^\a}{T^\a}\right)\frac
{q^2}{T^\a}\|u_n\|_\a^{4+\a} \le C_5 q^2 T^4.
\end{align*}
Then, from \eqref{eq:ineq} we deduce that
\begin{align}\label{eq:first}
\irt |\n u_n|^2 & \le 3 A_1 + (3 C_3+C_4q+C_5q)qT^4.
\end{align}
On the other hand, since $\langle(\Itqln)'(u_n),(u_n)\rangle=0$, by \eqref{eq:g1g2} we
have that
\begin{multline}\label{eq:neh}
\irt |\n u_n|^2 + q k_T(u_n)\irt\phi_nu^2_n+ \frac{q \a }{4 T^\a}
\chi'\left(\frac{\|u_n\|_\a^\a}{T^\a}\right)
\|u_n\|_\a^\a\irt \phi_nu_n^2\\
+ \irt
g_2(u_n)u_n = \l_n \irt g_1(u_n)u_n\le C_\eps\irt|u_n|^6+\eps\irt
g_2(u_n)u_n.
\end{multline}
Now, by \eqref{eq:g2} and \eqref{eq:neh}, we obtain
\begin{align}\label{eq:second}
m  (1-\eps)\irt u^2_n & \le (1-\eps)\irt g_2(u_n)u_n\nonumber\\
& \le C_\eps \irt
|u_n|^6 - \frac{q\a}{4T^\a} \chi'\left(\frac{\|u_n\|_\a^\a}{T^\a}\right)
\|u_n\|_\a^\a\irt \phi_nu_n^2\nonumber\\
&\le C \left(\irt |\n u_n|^2\right)^3 + \bar C q^2 T^4\nonumber\\
&\le C(3 A_1 + (3 C_3+C_4q+C_5q)q T^4)^3 + \bar C q^2 T^4
\end{align}
where in the last inequality we have used
\eqref{eq:first}.
\\
We suppose by contradiction that there exists no subsequence of
$(u_n)_n$ which is uniformly bounded by $T$ in the $\a-$norm. As a
consequence, for a certain $\bar n$ it should result that
\begin{equation}\label{eq:contr}
\|u_n\|_\a> T, \quad\forall n\ge\bar n.
\end{equation}
Without any loss of generality, we are supposing that
\eqref{eq:contr} is true for any $u_n.$
\\ 
Therefore, by \eqref{eq:first} and \eqref{eq:second}, we conclude
that
\begin{align*}
T^2 < \|u_n\|_\a^2 \le C \|u_n\|^2 \le C_6 +C_7
(q^3+ q^6) T^{12} + \bar C q^2 T^4
\end{align*}
which is not true for T large and q small enough.
\end{proof}

\begin{proofmain}
Let $T,q_0$ be as in Lemma \ref{le:T} and fix $0<q<q_0$. 
Let $u_n$ be a critical point for $\Itqln$ at level $c_{\l_n}$. We
prove that $(u_n)_n$ is a $H^1-$bounded Palais-Smale sequence for
$\Iq$.
\\
Since by Lemma \ref{le:T}
            \begin{equation}\label{eq:bb}
                \|u_n\|_\a\le T,
            \end{equation}
        the boundedness in the $H^1-$norm trivially follows from arguments such as those in
        \eqref{eq:first} and \eqref{eq:second}. Finally, by \eqref{eq:bb},
        certainly we have that
        $$\Itqln(u_n)=\frac 12 \irt |\n u|^2 + \frac{q}{4} \irt \phi_u
        u^2 +\irt G_2(u)
        -\l_n \irt G_1(u),$$
        and then, since $\l_n\nearrow 1$,
        we can prove that $(u_n)_n$ is a (PS) sequence for $\Iq$
        by similar argument as in \cite[Theorem 1.1]{AP}.
\\
Now we conclude arguing as in Lemma \ref{le:PS}.
\end{proofmain}

\appendix

\section{A Pohozaev type identity}

In this Section we show that if $u,\phi\in H^2_{loc}(\RT)$ solve
\begin{equation}\label{eq:poho}
\left\{
\begin{array}{ll}
\displaystyle{-\Delta u+q k_T(u) \phi u+q\frac{\a}{T^\a} \chi'\left(\frac{\|u\|_\a^\a}{T^\a}\right)|u|^{2/5}u\irt\phi u^2=g(u)}&\hbox{in }\RT,
\\
-\Delta \phi=q u^2&\hbox{in }\RT,
\end{array}
\right.
\end{equation}
then the following Pohozaev type identity
\begin{multline}\label{eq:pohoapp}
\frac{1}{2}\irt|\nabla u|^2 + \frac{5q}{4}k_T(u)\irt\phi u^2
+\frac{3q}{T^\a} \chi'\left(\frac{\|u\|_\a^\a}{T^\a}\right)\|u\|_\a^\a\irt\phi u^2
=3 \irt G(u)
\end{multline}
holds.

Indeed, by \cite[Lemma 3.1]{DM2}, for every $R>0$, we have
\begin{align}
\int_{B_R}\!\! -\Delta u (x\cdot\nabla u)&= - \frac{1}{2}\int_{B_R}\!\!|\nabla u|^2-
\frac{1}{R}\int_{\partial B_R} \!\!|x\cdot\nabla u|^2 + \frac{R}{2}\int_{\partial B_R}\!\! |\nabla u|^2, \label{eq:pohoBR1}
\\
\int_{B_R} \phi u (x\cdot\nabla u)&=- \frac{1}{2} \int_{B_R} u^2 (x\cdot \nabla\phi)- \frac{3}{2}\int_{B_R}\phi u^2
+\frac{R}{2}\int_{\partial B_R} \phi u^2, \label{eq:pohoBR2}
\\
\int_{B_R} g(u) (x\cdot\nabla u)&=-3\int_{B_R} G(u) + R\int_{\partial B_R} G(u), \label{eq:pohoBR3}
\\
\int_{B_R} \!\!|u|^{2/5}u (x\cdot\nabla u)&=-\frac{3}{\a}\int_{B_R}|u|^\a + \frac{R}{\a}\int_{\partial B_R}|u|^\a, \label{eq:pohoBR4}
\end{align}
where $B_R$ is the ball of $\RT$ centered in the origin and with radius $R$.

Multiplying the first equation of \eqref{eq:poho} by $x\cdot\nabla u$ and the second equation
by $x\cdot\nabla\phi$ and integrating on $B_R$, by \eqref{eq:pohoBR1}, \eqref{eq:pohoBR2}, \eqref{eq:pohoBR3} and \eqref{eq:pohoBR4} we get
\begin{align}
-&\; \frac{1}{2}\int_{B_R}|\nabla u|^2-
\frac{1}{R}\int_{\partial B_R} |x\cdot\nabla u|^2 + \frac{R}{2}\int_{\partial B_R} |\nabla u|^2 \nonumber
\\
&-\frac{q}{2}k_T(u) \int_{B_R} u^2 (x\cdot \nabla\phi)- \frac{3q}{2}k_T(u)\int_{B_R}\phi u^2
+\frac{R q}{2}k_T(u)\int_{\partial B_R} \phi u^2  \nonumber
\\
&-\frac{3q}{T^\a} \chi'\left(\frac{\|u\|_\a^\a}{T^\a}\right)\irt\phi u^2 \int_{B_R}|u|^\a
+\frac{R q}{T^\a} \chi'\left(\frac{\|u\|_\a^\a}{T^\a}\right)\irt\phi u^2 \int_{\partial B_R}|u|^\a  \nonumber
\\
=&\;-3\int_{B_R} G(u) + R\int_{\partial B_R} G(u) \label{eq:pohotBR1}
\end{align}
and
\begin{equation}\label{eq:pohotBR2}
q\int_{B_R} u^2 (x\cdot\nabla\phi)=- \frac{1}{2}\int_{B_R}|\nabla \phi|^2-
\frac{1}{R}\int_{\partial B_R} |x\cdot\nabla \phi|^2 + \frac{R}{2}\int_{\partial B_R} |\nabla \phi|^2.
\end{equation}

Substituting \eqref{eq:pohotBR2} into \eqref{eq:pohotBR1} we obtain
\begin{align*}
-& \;\frac{1}{2}\int_{B_R}|\nabla u|^2 - \frac{3q}{2}k_T(u)\int_{B_R}\phi u^2
+\frac{1}{4}k_T(u)\int_{B_R}|\nabla \phi|^2
\\
&-\frac{3q}{T^\a} \chi'\left(\frac{\|u\|_\a^\a}{T^\a}\right)\irt\phi u^2 \int_{B_R}|u|^\a
+3\int_{B_R} G(u)
\\
=&\;\frac{1}{R}\int_{\partial B_R} |x\cdot\nabla u|^2 - \frac{R}{2}\int_{\partial B_R} |\nabla u|^2
-\frac{1}{2R}k_T(u)\int_{\partial B_R} |x\cdot\nabla \phi|^2
\\
&+ \frac{R}{4}k_T(u)\int_{\partial B_R} |\nabla \phi|^2 -\frac{R q}{2}k_T(u)\int_{\partial B_R} \phi u^2
\\
&- \frac{R q}{T^\a} \chi'\left(\frac{\|u\|_\a^\a}{T^\a}\right)\irt\phi u^2 \int_{\partial B_R}|u|^\a
+ R\int_{\partial B_R} G(u).
\end{align*}
As in \cite{DM2}, the right hand side goes to zero as $R\to +\infty$ and so we get
\begin{multline*}
- \frac{1}{2}\irt|\nabla u|^2 - \frac{3q}{2}k_T(u)\irt\phi u^2
+\frac{1}{4}k_T(u)\irt|\nabla \phi|^2
\\
-\frac{3q}{T^\a} \chi'\left(\frac{\|u\|_\a^\a}{T^\a}\right)\irt\phi u^2 \irt|u|^\a
+3\irt G(u) =0.
\end{multline*}

If $(u,\phi_u)\in \H\times\D$ is a solution of \eqref{eq:poho}, by standard regularity results, $u,\phi_u\in H^2_{loc}(\RT)$  and, by  $i)$ of Lemma \ref{le:prop}, we get
\eqref{eq:pohoapp}.

\end{document}